\documentclass[11pt]{amsart}

\usepackage[paperheight=11in,paperwidth=8.5in,left=1in,top=1.25in,right=1in,bottom=1.25in]{geometry}
\usepackage[table]{xcolor}
\definecolor{Gray}{gray}{0.9}

\usepackage{amsmath, amsfonts, amssymb, amsthm}

\usepackage{hyperref}
 \hypersetup{
    colorlinks,
    citecolor=magenta,
    filecolor=magenta,
    linkcolor=teal,
    urlcolor=black
}

\usepackage{tikz}
\usetikzlibrary{arrows,calc,intersections,matrix,positioning,through}
\usepackage{tikz-cd}

\usepackage[capitalize,nameinlink,noabbrev,nosort]{cleveref}
\usepackage{float}
\usepackage{subfig}
\usepackage{url}

\usepackage[maxbibnames=99,sorting=nyt,giveninits,maxnames=10,backend=bibtex,block=space]{biblatex}

\usepackage{ytableau}

\usepgflibrary{shapes.geometric}
\usetikzlibrary{shapes.geometric}

\newtheorem{thm}{Theorem}[section]
\newtheorem{prop}[thm]{Proposition}
\newtheorem{coro}[thm]{Corollary}

\theoremstyle{definition}
\newtheorem{defi}[thm]{Definition}
\newtheorem{exam}[thm]{Example}
\theoremstyle{remark}
\newtheorem{rema}[thm]{Remark}

\newcommand \acknowledgements{\section*{Acknowledgements}}

\DeclareMathOperator{\DASEP}{DASEP}

\DeclareMathOperator{\CBP}{CBP}
\DeclareMathOperator{\RRG}{RRG}

\newcommand{\Q}{\mathbb Q}
\newcommand{\Z}{\mathbb Z}

\title[The doubly asymmetric simple exclusion process]{The doubly asymmetric simple exclusion process, the colored Boolean process, and the restricted random growth model}

\author{Yuhan Jiang}

\address{Harvard University\\ 
Dept. of Mathematics\\
1 Oxford St\\
Cambridge\\
MA 02138 (USA)}

\email{yjiang@math.harvard.edu}



\begin{document}

\begin{abstract}
The multispecies asymmetric simple exclusion process (mASEP) is a Markov chain in which particles of different species hop along a one-dimensional lattice. 
This paper studies the doubly asymmetric simple exclusion process $\DASEP(n,p,q)$ in which $q$ particles with species $1, \dots, p$ hop along a circular lattice with n sites, but also the particles are allowed to spontaneously change from one species to another. 
In this paper, we introduce two related Markov chains called the colored Boolean process and the restricted random growth model, and we show that the DASEP lumps to the colored Boolean process, and the colored Boolean process lumps to the restricted random growth model. 
This allows us to generalize a theorem of David Ash on the relations between sums of steady state probabilities. 
We also give explicit formulas for the stationary distribution of $\DASEP(n,2,2)$.
\end{abstract}

\maketitle

\section{Introduction}

The asymmetric simple exclusion process (ASEP) is a Markov chain for particles hopping on a one-dimensional lattice such that each site contains at most one particle. 
The ASEP was introduced independently in biology by Macdonald-Gibbs-Pipkin \cite{MGP}, and in mathematics by Spitzer \cite{Spitzer1970InteractionOM}. 
There are many versions of the ASEP: the lattice can have open boundaries, or be a ring, not necessarily finite (see Liggett \cite{Liggett1975ErgodicTF},\cite{Liggett1985InteractingParticleSystems}).
Particles can have different species, and this variation is called the multispecies ASEP (mASEP). 
The asymmetry can be partial, so that particles are allowed to hop both left and right, but one side is $t$ times more probable, and this is called the partially asymmetric exclusion process (PASEP). 
The ASEP is closely related to a growth model defined by Kardar-Parizi-Zhang \cite{Kardar1986DynamicSO}, and various methods have been invented to study the ASEP, such as the matrix ansatz introduced by Derrida et al. in \cite{Derrida_1993}. 
The combinatorics of the ASEP was studied by many people, see \cite{Duchi2004ACA}\cite{Blythe2004DyckPM}\cite{Brak2012SimpleAE}\cite{Martin2020}\cite{Corteel2018FromMQ}.

A partition $\lambda$ is a weakly decreasing sequence of $n$ nonnegative integers $\lambda = (\lambda_1 \geq \lambda_2 \geq \dots \geq \lambda_n \geq 0)$.
We denote the sum of all parts by $|\lambda| = \lambda_1 + \cdots + \lambda_n$.
We will write a partition as an $n$-tuple $\lambda = (\lambda_1, \lambda_2, \dots, \lambda_n)$.
Let $m_i = m_i(\lambda)$ be the number of parts of $\lambda$ that equal $i$. 
As in \cite[Section 7.2]{stanley_1999}, we also denote a partition by $\lambda = \langle 1^{m_1}2^{m_2}\cdots \rangle$. 
Let $\ell(\lambda)$ denote the number of nonzero (positive) parts of $\lambda$, or the length of $\lambda$. 
We have $\ell(\lambda) = \sum_i m_i(\lambda)$.
We write $S_n(\lambda)$ as the set of all permutations of $(\lambda_1, \dots, \lambda_n)$.
The mASEP can be thought of a Markov chain on $S_n(\lambda)$.

Let $n$ be the number of positions on the lattice, $p$ be number of types of species, and $q$ be the number of particles.
David Ash~\cite{ash2023introducing} defined the \emph{doubly asymmetric simple exclusion process} $\DASEP(n,p,q)$. 
The DASEP is a variant of the mASEP but also allows particles to change from one species to another. If $p=1$, $\DASEP(n,1,q)$ is the usual 1-species PASEP on a ring.

\begin{defi}\cite{ash2023introducing}
For all positive integers $n,p,q$ with $n > q$, the \emph{doubly asymmetric simple exclusion process} $\DASEP(n,p,q)$ is a Markov process on the following set
$$\Gamma^{p,q}_n = \bigcup_{\substack{\lambda_1 \leq p, \\ \ell(\lambda) = q}} S_n(\lambda).$$
Let $0 \le t,u \le 1$ be constants.
The transition probability $P(\mu,\nu)$ on two permutations $\mu$ and $\nu$ is as follows:
\begin{itemize}
    \item If $\mu = (\mu_1, \dots, \mu_k, i, j, \mu_{k+2}, \dots, \mu_n)$ and $\nu = (\mu_1, \dots, \mu_k, j, i, \mu_{k+2}, \dots, \mu_n)$ with $i \neq j$, then $P(\mu,\nu) = \frac{t}{3n}$ if $i>j$ and $P(\mu,\nu) = \frac{1}{3n}$ if $j>i$.
    \item If $\mu = (i, \mu_2, \mu_3, \dots, \mu_{n-1}, j)$ and $\nu = (j, \mu_2, \mu_3, \dots, \mu_{n-1}, i)$ with $i \neq j$, then $P(\mu,\nu) = \frac{t}{3n}$ if $j>i$ and $P(\mu,\nu) = \frac{1}{3n}$ if $i>j$.
    \item If $\mu = (\mu_1, \dots, \mu_k, i, \mu_{k+2}, \dots, \mu_n)$ and $\nu = (\mu_1, \dots, \mu_k, i+1, \mu_{k+2}, \dots, \mu_n)$ with $i\leq p-1$, then $P(\mu,\nu)=\frac{u}{3n}$.
    \item If $\mu = (\mu_1, \dots, \mu_k, i+1, \mu_{k+2}, \dots, \mu_n)$ and $\nu = (\mu_1, \dots, \mu_k, i, \mu_{k+2}, \dots, \mu_n)$ with $i\geq1$, then $P(\mu,\nu)=\frac{1}{3n}$.
    \item If none of the above conditions apply but $\nu \neq \mu$ then $P(\mu,\nu)=0$. If $\nu = \mu$ then $P(\mu,\mu) = 1 - \sum_{\nu \neq \mu} P(\mu,\nu)$.
\end{itemize}

As we are interested in the stationary distribution of DASEP, we denote the un-normalised stationary distribution by $\{\pi_{\DASEP}(\mu): \mu \in S_n(\lambda)\}$, which is uniquely defined if we require the $p_\mu$ to be polynomials with greatest common divisor equal to 1.
\end{defi}

\begin{rema}\label{remark:cyclic}
There is an inherent cyclic symmetry in the definition of DASEP.
\end{rema}

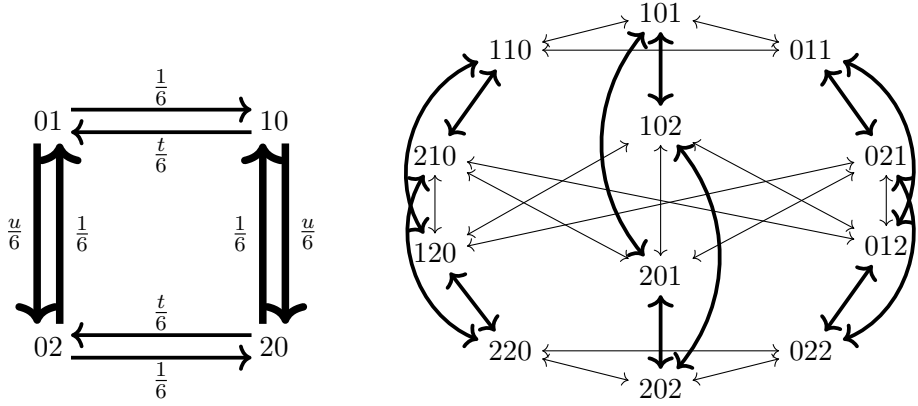
\begin{figure}
\centering
\begin{tikzpicture}[scale=1.5]
    \draw (-1,1) node {01};
    \draw[line width = .5mm, ->] (-.8,1.1) -- (.8,1.1);
    \draw (0,1.3) node {$\frac16$};
    \draw[line width = .5mm, <-] (-.8,.9) -- (.8,.9);
    \draw (0,.7) node {$\frac{t}{6}$};

    \draw[line width = 1mm, ->] (-1.1,.8) -- (-1.1,-0.8);
    \draw (-1.3,0) node {$\frac{u}{6}$};
    \draw[line width = 1mm, <-] (-.9,.8) -- (-.9,-0.8);
    \draw (-.7,0) node {$\frac16$};
    \draw (1,1) node {10};
    \draw (-1,-1) node {02};
    \draw (1,-1) node {20};

    \draw[line width = 1mm, ->] (1.1,.8) -- (1.1,-0.8);
    \draw (1.3,0) node {$\frac{u}{6}$};
    \draw[line width = 1mm, <-] (.9,.8) -- (.9,-0.8);
    \draw (.7,0) node {$\frac16$};
    
    \draw[line width = .5mm, ->] (-.8,-1.1) -- (.8,-1.1);
    \draw (0,-1.3) node {$\frac16$};
    \draw[line width = .5mm, <-] (-.8,-.9) -- (.8,-.9);
    \draw (0,-.7) node {$\frac{t}{6}$};
\end{tikzpicture}
\hspace{.5cm}
\begin{tikzpicture}[scale = 2]
    \draw (-1,.5) node (bbw) {110};
    \draw (0,.75) node (bwb) {101};
    \draw (1,.5) node (wbb) {011};

    \draw[<->] (bbw) -- (bwb);
    \draw[<->] (bwb) -- (wbb);
    \draw[<->] (bbw) -- (wbb);

    \draw (0,0) node (21bwb) {102};
    \draw (0,-1) node (12bwb) {201};
    \draw (-1.5, -.2) node (21bbw) {210};
    \draw (-1.5, -.85) node (12bbw) {120};
    \draw (1.5, -.2) node (21wbb) {021};
    \draw (1.5, -.8) node (12wbb) {012};

    \draw[<->] (21bbw) -- (12bbw);
    \draw[<->] (21bwb) -- (12bwb);
    \draw[<->] (21wbb) -- (12wbb);

    \draw[<->] (12bbw) -- (21bwb);
    \draw[<->] (21bbw) -- (12bwb);

    \draw[<->] (21bwb) -- (12wbb);
    \draw[<->] (12bwb) -- (21wbb);

    \draw[<->] (21bbw) -- (12wbb);
    \draw[<->] (12bbw) -- (21wbb);

    \draw (-1,-1.5) node (22bbw) {220};
    \draw (0, -1.75) node (22bwb) {202};
    \draw (1, -1.5) node (22wbb) {022};

    \draw[<->] (22bbw) -- (22bwb);
    \draw[<->] (22bwb) -- (22wbb);
    \draw[<->] (22bbw) -- (22wbb);

    \draw[line width = .5mm, <->] (bbw) -- (21bbw);
    \draw[line width = .5mm, <->] (bbw) to [bend right=50] (12bbw);

    \draw[line width = .5mm, <->] (wbb) -- (21wbb);
    \draw[line width = .5mm, <->] (wbb) to [bend left=50] (12wbb);

    \draw[line width = .5mm, <->] (bwb) -- (21bwb);
    \draw[line width = .5mm, <->] (bwb) to [bend right=40] (12bwb);

    \draw[line width = .5mm, <->] (22bbw) -- (12bbw);
    \draw[line width = .5mm, <->] (22bbw) to [bend left=50] (21bbw);

    \draw[line width = .5mm, <->] (22bwb) -- (12bwb);
    \draw[line width = .5mm, <->] (22bwb) to [bend right=40] (21bwb);

    \draw[line width = .5mm, <->] (22wbb) -- (12wbb);
    \draw[line width = .5mm, <->] (22wbb) to [bend right=50] (21wbb);
\end{tikzpicture}
    \caption{The state diagram of $\DASEP(2,2,1)$ and $\DASEP(3,2,2)$. We omit loops at each state. Bold edges denote changes in species, while regular edges denote exchanges of particles of different species or between particles and holes.}
    \label{fig:dasep221322}
\end{figure}

Our first main result is about the ratio between the sums of certain groups of $p_\mu$.
For each partition $\lambda$ and each binary word $w = (w_1, \dots, w_n) \in S_n(1^q 0^{n-q})$, define 
$$S_n^w(\lambda) := \{ \mu \in S_n(\lambda)| \mu_i \neq 0 \text{ if and only if } w_i \neq 0 \}$$
as the set of all permutations of $\lambda$ whose zeros are aligned with the binary word $w$.
Then if $\lambda_1 \leq p$ and $\ell(\lambda) = q$, we have $|S_n(\lambda)| = \binom{n}{n-q,m_1, \dots, m_p}$ and $|S_n^w(\lambda)| = \binom{q}{m_1, m_2, \dots, m_p}$.

\begin{exam}
For the partition $\lambda=(2,1)$, we have $|\lambda| = 2+1 = 3$. For $n=3$, we have $|S_3^{011}((2,1))| = \binom{2}{1,1} = 2$. For $n=4$, we have $|S_4((2,1))| = \binom{4}{2,1,1} = 12$.
\end{exam}

\begin{thm}\label{thm:main}
Consider $\DASEP(n,p,q)$ for any positive integers $n,p,q$ with $n >q$. 
\begin{enumerate}
    \item For any two binary words $w, w' \in S_n(1^q 0^{n-q})$, we have $\pi_{\DASEP}(w) = \pi_{\DASEP}(w')$.
    \item For any binary word $w \in S_n(1^q 0^{n-q})$ and partition $\lambda = \langle 1^{m_1} 2^{m_2} \cdots p^{m_p} \rangle$ such that 
    $$m_1 + \cdots + m_p = q,$$
    we have
    \begin{align*}
        \sum_{\mu \in S_n^w(\lambda)} \pi_{\DASEP}(\mu) &= u^{|\lambda|-q} |S_n^w(\lambda)| \pi_{\DASEP}(w) \\
        &= u^{|\lambda|-q} \binom{q}{m_1, m_2, \dots, m_p} \pi_{\DASEP}(w).
    \end{align*}
\end{enumerate}

In other words, the sum of steady state probabilities of states within one equivalence class is proportional to each other and the ratio only depends on the sum of all parts and the multiplicities in the partition; all steady state probabilities with respect to binary words are equal.
\end{thm}

\begin{rema}
    In the special case of $\DASEP(3,p,2)$, \cref{thm:main} was proved by David Ash \cite[Theorem 5.2]{ash2023introducing}.
\end{rema}

\begin{table}
\centering
    \begin{tabular}{|c|c|}
    \hline
       $\mu$ & $\pi_{\DASEP}(\mu)$ \\ \hline
        011 & $u+3t+4$ \\ \hline
        \rowcolor{Gray}
        012 & $u(u+4t+3)$ \\ \hline
        \rowcolor{Gray}
        021 & $u(u+2t+5)$ \\ \hline
        022 & $u^2(u+3t+4)$ \\ \hline
    \end{tabular}
    \hspace{3cm}
    \begin{tabular}{|c|c|}
    \hline
        $\mu$ & $\pi_{\DASEP}(\mu)$ \\ \hline
        0011 & $u+2t+3$ \\ \hline
        0101 & $u+2t+3$ \\ \hline
        0022 & $u^2(u+2t+3)$ \\ \hline
        0202 & $u^2(u+2t+3)$ \\ \hline 
        \rowcolor{Gray}
        0012 & $u(u+3t+2)$ \\ \hline
        0102 & $u(u+2t+3)$ \\ \hline
        \rowcolor{Gray}
        0021 & $u(u+t+4)$ \\ \hline
    \end{tabular}
    \caption{Unnormalized steady state probabilities of $\DASEP(3,2,2)$ and $\DASEP(4,2,2)$. We present all states up to cyclic symmetry. The shaded rows on the left belong to $S_3^{011}((2,1))$, and on the right belong to $S_4^{0011}((2,1))$.}
    \label{tab:3&4}
\end{table}

\begin{exam}
    In $\DASEP(3,2,2)$, we compute the steady state probabilities for the binary word $w = 011$ and permutations of $(1,1,0), (2,1,0), (2,2,0)$ aligned with $w$ as shown in \cref{tab:3&4}.
    By \cref{thm:main}, we have $\pi_{\DASEP}(012) + \pi_{\DASEP}(021) = 2u \pi_{\DASEP}(011)$ and $\pi_{\DASEP}(022) = u^2 \pi_{\DASEP}(011)$, which can be seen from \cref{tab:3&4}.

    Similarly, in $\DASEP(4,2,2)$, we compute the steady state probabilities for binary words $0011$ and $0101$ in \cref{tab:3&4}. 
    We have $\pi_{\DASEP}(0011) = \pi_{\DASEP}(0101)$ and $\pi_{\DASEP}(0012) + \pi_{\DASEP}(0021) = 2u \pi_{\DASEP}(0011)$. 
    Since $0201$ is a cyclic permutation of $0102$, by \cref{remark:cyclic}, we have $\pi_{\DASEP}(0201) = \pi_{\DASEP}(0102)$. Therefore, we can see from \cref{tab:3&4} that $\pi_{\DASEP}(0102) + \pi_{\DASEP}(0201) = 2u \pi_{\DASEP}(0101)$ as asserted by \cref{thm:main}.
\end{exam}

To prove \cref{thm:main}, we also introduce a new Markov chain that we call the \emph{colored Boolean process} (see \cref{def:cBp}), and we show that the DASEP \emph{lumps} to the colored Boolean process. 
This gives a relationship between the stationary distribution of colored Boolean process and the DASEP, see \cref{thm:relationship}.

\begin{coro}\label{cor:restricted}
    For the $\DASEP(n,p,q)$ defined by positive integers $n,p,q, n > q$, and $\lambda, \mu$ two partitions with $\lambda_1 \leq p, \mu_1 \leq p, \ell(\lambda) = \ell(\mu) = q$, we have
    $$\frac{\sum_{\nu\in S_n(\lambda)} \pi_{\DASEP}(\nu)}{\sum_{\nu \in S_n(\mu)} \pi_{\DASEP}(\nu)} = \frac{|S_n(\lambda)|}{|S_n(\mu)|} u^{|\lambda| - |\mu|}.$$
\end{coro}

In \cref{thm:n22}, we also give explicit formulas for the stationary distributions of the infinite family $\DASEP(n,2,2)$ for any $n \geq 3$. The formulas depend on whether $n$ is odd or even. Both are described by polynomial sequences given by a second-order homogeneous recurrence relation (see \cref{thm:n22}).

When we specialize to $u=t=1$, the polynomial sequences specialize to the trinomial transform of Lucas number \href{https://oeis.org/A082762}{A082762} or the binomial transform of the denominators of continued fraction convergents to $\sqrt{5}$ \href{https://oeis.org/A084326}{A084326} \cite{oeis}.

The structure of this paper is as follows. In \cref{daseptocbp}, we define the \emph{colored Boolean process}, and we show that the DASEP \emph{lumps} to the colored Boolean process. In \cref{cbptorgm}, we define the \emph{restricted random growth model}, which is a Markov chain on Young diagrams, and we show that the colored Boolean process lumps to the restricted random growth model. In \cref{n22section}, we give explicit formulas for the stationary distributions of the infinite family $\DASEP(n,2,2)$.

\acknowledgements{We thank Lauren Williams for suggesting the problem and helping me better understand ASEP.}

\section{The DASEP lumps to the colored Boolean process}\label{daseptocbp}

In this section, we define the \emph{colored Boolean process}, and we show that the DASEP \emph{lumps} to the colored Boolean process. 
We compute the ratios between steady states probabilities in the colored Boolean process, leading us to prove \cref{thm:main}.

\begin{defi}\label{def:cBp}
The \emph{colored Boolean process} is a Markov chain dependent on three positive integers $n,p,q$ with $n>q$ on states space
$$\Omega^{p,q}_n = \{(w, \lambda) | w \in S_n(1^q 0^{n-q}), \lambda_1 \leq p, \ell(\lambda) = q \} $$
with the following transition probabilities:
\begin{itemize}
    \item $Q((w,\lambda),(w,\lambda')) = \frac{m_iu}{3n}$ if $\lambda$ contains $m_i \geq 1$ parts equal to $i$ and $\lambda'$ is obtained from $\lambda$ by changing a part equal to $i$ to a part equal to $i+1$.
    \item $Q((w,\lambda),(w,\lambda')) = \frac{m_i}{3n}$ if $\lambda$ contains $m_i \geq 1$ parts equal to $i$ and $\lambda'$ is obtained from $\lambda$ by changing a part equal to $i$ to a part equal to $i-1$.
    \item $Q((w,\lambda),(w',\lambda)) = \frac{1}{3n}$ if $w'$ is obtained from $w$ by $01 \to 10$ at a unique position (allowing wrap-around at the end).
    \item $Q((w,\lambda),(w',\lambda)) = \frac{t}{3n}$ if $w'$ is obtained from $w$ by $10 \to 01$ at a unique position (allowing wrap-around at the end).
    \item If none of the above conditions apply but $w \neq w'$ or $\lambda \neq \lambda'$, then $Q((w,\lambda),(w',\lambda')) = 0$. 
    \item $Q((w,\lambda),(w,\lambda)) = 1-\sum_{(w',\lambda') \neq (w,\lambda)} Q((w,\lambda),(w',\lambda')).$
\end{itemize}

We denote the stationary distribution of $\Omega^{p,q}_n$ by $\pi_{\CBP}$.
\end{defi}

We think of parts of different sizes as particles of different colors, or species, hence the name.

The relation between the colored Boolean process and the DASEP is captured by the following notion.
\begin{defi}\cite[Section 6.3]{kemeny}\cite[Definition 2.5, Theorem 2.6]{Pang2015LumpingsOA}\label{def:lumping}
Let $\{X_t\}$ be a Markov chain on state space $\Omega_X$ with transition matrix $P$, and let $f: \Omega_X \to \Omega_Y$ be a surjective map. Suppose there is an $|\Omega_Y| \times |\Omega_Y|$ matrix $Q$ such that for all $y_0, y_1 \in \Omega_Y$, if $f(x_0) = y_0$, then
$$ \sum_{x: f(x) = y_1} P(x_0, x) = Q(y_0, y_1). $$
Then $\{f(X_t)\}$ is a Markov chain on $\Omega_Y$ with transition matrix $Q$. We say that $\{f(X_t)\}$ is a \emph{lumping} of $\{X_t\}$ and $\{X_t\}$ is a \emph{lift} of $\{f(X_t)\}$.
\end{defi}

\begin{figure}
\centering
\includegraphics{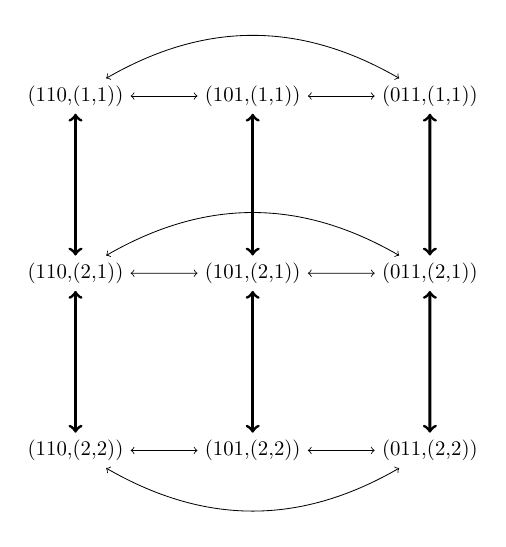}
\caption{The state diagram of $\Omega^{2,2}_3$, as a lumping of $\DASEP(3,2,2)$ as in \cref{fig:dasep221322}. The bold edges denote the changes of species, while the regular edges denote the exchanges between particles of different species or between particles and holes.}
\label{fig:omega322}
\end{figure}

\begin{thm}\label{thm:relationship}
The projection map $f: \Gamma^{p,q}_n \to \Omega^{p,q}_n$ sending each $\mu \in S_n^w(\lambda)$ to $(w,\lambda)$ is a lumping of $\DASEP(n,p,q)$ onto the colored Boolean process $\Omega^{p,q}_n$. 
\end{thm}

\begin{proof}
Fix $(w_0,\lambda_0)$ and $(w_1, \lambda_1)$, we want to show that for any $\mu_0 \in S_n^{w_0}(\lambda_0)$, the quantity
$$ \sum_{\mu: f(\mu) \in S_n^{w_1}(\lambda_1)} P(\mu_0, \mu) $$ is independent of the choice of $\mu_0$ and equal to $Q((w_0,\lambda_0),(w_1,\lambda_1))$. We may assume $(w_0, \lambda_0) \neq (w_1, \lambda_1)$, because the probabilities add up to 1.

This quantity is nonzero only in the following cases:
\begin{itemize}
    \item $w_0 = w_1$ and if $\lambda_0 = 1^{m_1}2^{m_2}\cdots p^{m_p}$ there exists a unique $i \in [1, p-1]$ with $m_i \geq 1$ such that $\lambda_1 = \cdots i^{m_i-1} (i+1)^{m_{i+1} +1} \cdots$. We upgrade the species of a particle from $i$ to $i+1$, and there are again $m_i$ ways to do it, where each transition probability $P(\mu_0, \mu) = \frac{u}{3n}$, so the quantity is equal to $\frac{m_i u}{3n}$.
    \item $w_0 = w_1$ and if $\lambda_0 = 1^{m_1}2^{m_2}\cdots p^{m_p}$ there exists a unique $i \in [2,p]$ with $m_i \geq 1$ such that $\lambda_1 = \cdots (i-1)^{m_{i-1} + 1} i^{m_i - 1} \cdots$.
    We downgrade the species of a particle from $i$ to $i-1$, and there are $m_i$ ways to do it, so there are $m_i$ number of $\mu$'s in $S_n^{w_1}(\lambda_1)$ with nonzero $P(\mu_0, \mu) = \frac{1}{3n}$, so the quantity is equal to $\frac{m_i}{3n}$.
    \item $\lambda_0 = \lambda_1$ and $w_1$ is obtained from $w_0$ by $01 \to 10$ at a unique position (allowing wrap-around at the end).
    This quantity is equal to $\frac{1}{3n}$.
    \item $\lambda_0 = \lambda_1$ and $w_1$ is obtained from $w_0$ by $10 \to 01$ at a unique position (allowing wrap-around at the end).
    This quantity is equal to $\frac{t}{3n}$.
\end{itemize}

The nonzero transition probabilities of $\Omega$ in each of the four cases above is the same as we defined.
\end{proof}

We may use the stationary distribution of $\{X_t\}$ to compute that of its lumping.
\begin{prop}\cite[Section 6.3]{kemeny}\label{prop:lumping}
    Suppose $p$ is a stationary distribution for $\{X_t\}$, and let $\pi$ be the measure on $\Omega_Y$ defined by $\pi(y) = \sum_{x: f(x)=y} p(x)$. Then $\pi$ is a stationary distribution for $\{f(X_t)\}$.
\end{prop}
Thus, we may use the stationary distribution of $\Omega^{p,q}_n$ to study that of $\Gamma_n^{p,q}$. We have the following corollary due to \cref{prop:lumping} and \cref{thm:relationship}.

\begin{coro}
    The unnormalized steady state probabilities $\{\pi_{\CBP}(w,\lambda)| (w,\lambda) \in \Omega^{p,q}_n\}$ of the colored Boolean process and the unnormalized steady state probabilities $\{\pi_{\DASEP}(\mu)|\mu \in \Gamma^{p,q}_n\}$ of the DASEP are related as follows:
$$\pi_{\CBP}(w,\lambda) \propto \sum_{\mu \in S_n^w(\lambda)} \pi_{\DASEP}(\mu).$$
\end{coro}

\begin{thm}\label{thm:equal}
Consider the colored Boolean process $\Omega^{p,q}_n$.
\begin{enumerate}
    \item The steady state probabilities of all binary words are equal, i.e., 
    $$\pi_{\CBP}(w, 1^q 0^{n-q}) = \pi_{\CBP}(w', 1^q 0^{n-q})$$ for any $w, w' \in S_n(1^q 0^{n-q})$.
    \item The steady state probability $\pi_{\CBP}(w,\lambda)$ of an arbitrary state $(w, \lambda)$ can be expressed in terms of the steady state probability $\pi_{\CBP}(w,1^q 0^{n-q})$ (of the corresponding state $(w, 1^q 0^{n-q})$ with the trivial partition $(1^q 0^{n-q})$) as follows:
    \begin{equation}\label{eq:thmequal}
        \pi_{\CBP}(w,\lambda) = u^{|\lambda|-q} \binom{q}{m_1, \dots, m_p} \pi_{\CBP}(w, 1^q 0^{n-q}).
    \end{equation}
\end{enumerate}
\end{thm}

\begin{proof}
The colored Boolean process is again an ergodic markov chain, so we only need to show that the above relations satisfy the balance equations given by the transition matrix of $\Omega$.

Let us first check it for $\lambda = 1^q 0^{n-q}$. 
For any binary word $w$, denote $\pi_{\CBP}(w, 1^q 0^{n-q})$ by $p_w$. 
Let $b_w$ be the number of blocks of consecutive $1$'s in $w$ (allowing wrap-around). 
Notice that any occurrence of $01$ in $w$ must begin a block, and any occurrence of $10$ must signify the end of a block. We have
\begin{equation}\label{balance}
    (qu + b_w + b_w t) p_w = q \pi_{\CBP}(w, 1^{q-1}2) + b_w t \sum_{\underset{10\to01}{w' \to w}} p_{w'} + b_w \sum_{\underset{01 \to 10}{w'' \to w}} p_{w''}.
\end{equation}
Expanding the multinomial coefficient, we are left with
$$b_w(1 + t) p_w = b_w t \sum_{\underset{10\to01}{w' \to w}} p_{w'} + b_w \sum_{\underset{01 \to 10}{w'' \to w}} p_{w''} $$
which will be satisfied if we set all $p_w$'s to be equal.

For a generic $\lambda = \langle 1^{m_1}2^{m_2}\cdots p^{m_p} \rangle$, \cref{balance} is modified on the left hand side such that $qu$ becomes $au + b$ where $a = m_1 + \cdots + m_{p-1}$ and $b = m_2 + \cdots + m_p$. 
On the right hand side pf \cref{balance}, we change the first term to
{\small
$$\sum_{i < p} (m_{i+1} + 1) \pi_{\CBP}(w, \cdots i^{m_i-1} (i+1)^{m_{i+1} +1} \cdots) + \sum_{i > 1} (m_{i-1}+1)u \pi_{\CBP}(w, \cdots (i-1)^{m_{i-1} + 1} i^{m_i - 1} \cdots).$$
}

Expand this using \cref{eq:thmequal}, the multinomial coefficients gives the ratios
\begin{align*}
    \frac{\pi_{\CBP}(w, \cdots i^{m_i-1} (i+1)^{m_{i+1} +1} \cdots)}{\pi_{\CBP}(w, 1^{m_1}2^{m_2}\cdots p^{m_p})} &= \frac{m_iu}{m_{i+1}+1} \\
    \frac{\pi_{\CBP}(w, \cdots (i-1)^{m_{i-1}+1} i^{m_i-1} \cdots)}{\pi_{\CBP}(w, 1^{m_1}2^{m_2}\cdots p^{m_p})} &= \frac{m_i}{(m_{i-1}+1)u}
\end{align*}
for all $i$. Then we have a term by term equality for each $i$ where $a$ corresponds to the first summation and $b$ corresponds to the second.
\end{proof}

\begin{proof}[Proof of \cref{thm:main}]
    This follows directly from \cref{prop:lumping}, \cref{thm:relationship} and \cref{thm:equal}.
\end{proof}

\section{The colored Boolean process lumps to the restricted random growth model}\label{cbptorgm}

In this section, we define the \emph{restricted random growth model}, which is a Markov chain on the set of Young diagrams fitting inside a rectangle, and we show it to be a lumping of the colored Boolean process.

For partitions $\nu$ and $\lambda$, we write $\lambda \gtrdot_i \nu$ if there exists a unique $j$ such that $\lambda_j = \nu_j + 1 = i$ and for all $k \neq j$ we have $\lambda_k = \nu_k$. We write $\nu \lessdot_i \lambda$ if there exists a unique $j$ such that $\nu_j = \lambda_j - 1 = i$ and for all $k \neq j$ we have $\nu_k = \lambda_k$.
In both cases, we have $m_i(\nu) = m_i(\lambda)+1$ where $m_i(\nu)$ is the number of parts of $\nu$ equal to $i$.

\begin{exam}
    $(2,\textcolor{magenta}{2},1,0,0) \gtrdot_2 (2,1,1,0,0)$ and $(2,\textcolor{magenta}{1},1,0,0) \lessdot_1 (2,2,1,0,0)$.
\end{exam}

\begin{defi}
    Define the \emph{restricted random growth model} on the the state space $\chi^{p,q} = \{\lambda: \lambda_1 \leq p, \ell(\lambda) = q\}$ which includes all partitions that fit inside a $q \times p$ rectangle but do not fit inside a shorter rectangle, with transition probabilities $d^{(n)}_{\nu, \lambda}$ as follows:
    \begin{itemize}
        \item If $\nu \lessdot_i \lambda$, then $d^{(n)}_{\nu,\lambda} = \frac{m_i(\nu) u}{3n}$.
        \item If $\nu \gtrdot_i \lambda$, then $d^{(n)}_{\nu, \lambda} = \frac{m_i(\nu)}{3n}$.
        \item In all other cases where $\nu \neq \lambda$, $d^{(n)}_{\nu, \lambda} = 0$.
        \item $d^{(n)}_{\lambda,\lambda} = 1-\sum_{\nu : \nu \neq \lambda} d^{(n)}_{\nu,\lambda}$.
    \end{itemize}
    Recall that we also denote the number of parts of the partition $\nu$ that equal to $i$ by $m_i(\nu)$.
    We denote the unnormalized steady state probabilities of the restricted random growth model by $\pi_{\RRG}$.
\end{defi}

\begin{figure}[h]
    \centering
    \begin{tikzpicture}[scale=1.5]
        \draw (0,1.5) node (11) {{\tiny \ydiagram{1,1}}};
        \draw (0,0) node (21) {{\tiny\ydiagram{2,1}}};
        \draw (0,-1.5) node (21) {\tiny{\ydiagram{2,2}}};
        
        \draw[line width = .5mm, ->] (-0.2, 1) -- (-.2,.5);
        \draw (-.5, .75) node {$2u$};
        \draw[<-] (.2, 1) -- (.2, .5);
        \draw (.4, .75) node {$1$};

        \draw[line width = .5mm, ->] (-0.2, -.5) -- (-.2,-1);
        \draw (-.4, -.75) node {$u$};
        \draw[<-] (.2, -.5) -- (.2, -1);
        \draw (.4, -.75) node {$2$};
    \end{tikzpicture}
    \hspace{2cm}
    \begin{tikzpicture}[scale=1.5]
        \draw (0,1.5) node (11) {{\tiny \ydiagram{1,1}}};
        \draw (-1,0) node (21) {{\tiny \ydiagram{2,1}}};
        \draw (0,0) node {$\cong$};
        \draw (1,0) node (12) {{\tiny \ydiagram{1,2}}};
        \draw (0,-1.5) node (22) {{\tiny \ydiagram{2,2}}};

        \draw[line width = .5mm, ->] (-.4, 1) -- (-1.4,.5);
        \draw[line width = .5mm, ->] (.4, 1) -- (1.4,.5);
        \draw (-1, .9) node {$u$};
        \draw (1, .9) node {$u$};
        \draw[<-] (-.1, 1) -- (-1.1,.5);
        \draw[<-] (.1, 1) -- (1.1,.5);
        \draw (-.4, .5) node {$1$};
        \draw (.4, .5) node {$1$};

        \draw[line width = .5mm, ->] (-1.4, -.5) -- (-.4,-1);
        \draw[line width = .5mm, ->] (1.4, -.5) -- (.4,-1);
        \draw (-1, -.9) node {$u$};
        \draw (1, -.9) node {$u$};
        \draw[<-] (-1.1, -.5) -- (-.1, -1);
        \draw[<-] (1.1, -.5) -- (.1, -1);
        \draw (-.4, -.5) node {$1$};
        \draw (.4, -.5) node {$1$};
    \end{tikzpicture}
    \caption{The state diagram of the restricted random growth model on $\chi^{2,2}$. The Markov chain on the left can be viewed as a lumping of the Markov chain on the right in which we do not rearrange the parts in weakly decreasing order.}
    \label{fig:restricted}
\end{figure}
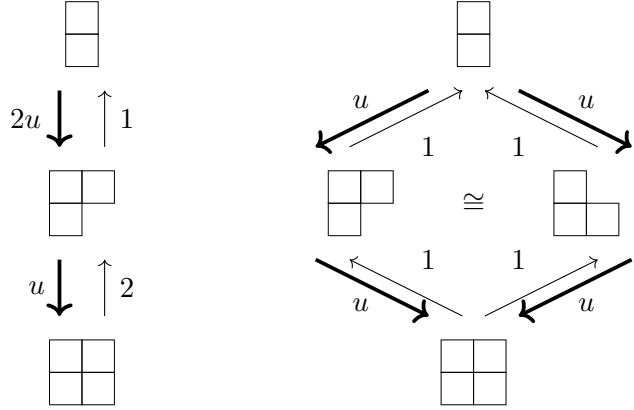

\begin{rema}
    In other words, the transitions randomly add or remove a box from the right of a random chosen row of the Young diagram of the partition (conditioned on rightly fitting in the $q \times p$ rectangle) as shown on the right hand side of \cref{fig:restricted}, then rearrange the parts in weakly decreasing order as shown on the left hand side of \cref{fig:restricted}.

    The 1d random growth model~\cite{Ferrari2010RandomGM} is a growth model on diagrams not necessarily arranged in weakly decreasing order. The diagrams do not need to fit in a rectangle, and the times between arrivals are independent i.i.d..
\end{rema}

\begin{thm}\label{2ndlumping}
    The projection map on state spaces $g: \Omega^{p,q}_n \to \chi^{p,q}$ sending $(w,\lambda)$ to $\lambda$ (forgetting the positions of 0's) is a lumping of the colored Boolean process $\Omega^{p,q}_n$ to the restricted random growth $\chi^{p,q}$.
\end{thm}

\begin{proof}
    By \cref{def:lumping}, we need to show that for any $\nu, \lambda$ and binary word $w$ the following equation holds:
    $$d^{(n)}_{\nu,\lambda} = \sum_{w'} Q((w,\nu),(w',\lambda)).$$

    It suffices to check this for all $\nu \neq \lambda$.
    Then $Q((w,\nu),(w',\lambda)) \neq 0$ only if $w = w'$ by \cref{def:cBp}, and this quantity is either $\frac{m_i(\nu) u}{3n}$ when $\nu \lessdot_i \lambda$ or $\frac{m_i(\nu)}{3n}$ when $\nu \gtrdot_i \lambda$.
\end{proof}

\begin{coro}
    The steady state probabilities $\{\pi_{\RRG}(\lambda)| \lambda \in \chi^{p,q} \}$ of the restricted random growth and the unnormalized steady state probabilities $\{\pi_{\CBP}(w,\lambda) |(w, \lambda) \in \Omega^{p,q}_n \}$ of the colored Boolean process are related as follows:
    $$\pi_{\RRG}(\lambda) \propto \sum_{w \in S_n(1^q 0^{n-q})} \pi_{\CBP}(w,\lambda).$$
\end{coro}

\begin{thm}\label{thm:multinomial}
The steady state probabilities $\{\pi_{\RRG}(\lambda)|\lambda \in \chi^{p,q} \}$ of the restricted random growth model satisfy the following relations for all partitions $\mu, \lambda \in \chi^{p,q}$:
    $$\frac{\pi_{\RRG}(\lambda)}{\pi_{\RRG}(\mu)} = \frac{|S_n(\lambda)|}{|S_n(\mu)|} u^{|\lambda| - |\mu|}.$$
\end{thm}

\begin{proof}
    This follows from \cref{2ndlumping} and \cref{thm:equal} and a computation on multinomial coefficients.
\end{proof}

\begin{proof}[Proof of \cref{cor:restricted}]
    This follows from \cref{prop:lumping}, \cref{2ndlumping} and \cref{thm:multinomial}.
\end{proof}

\section{The stationary distribution of DASEP(n,2,2)}\label{n22section}

If there were only one species of particle, i.e. $p=1$, the stationary distribution of $\DASEP(n,1,q)$ is uniform.
If there were only one particle, i.e., $q=1$, then the unnormalized steady state probabilities of $\DASEP(n,p,1)$ are given by powers of $u$, not involving $t$ due to cyclic symmetry.
We now study the first nontrivial case of $\DASEP(n,p,q)$ in more detail, namely $\DASEP(n,2,2)$.
In this section, we give a complete description of the stationary distributions when there are two particles and two species, while the number of sites can be arbitrary.

\begin{thm}\label{thm:n22}
    Let $(a_k)_{k\geq0}$ and $(b_k)_{k\geq-1}$ be polynomial sequences in $u,t$ satisfying the recurrence relation
    \begin{align}
        \label{recurrence1} a_k &= (u+2t+3)a_{k-1}-(t+1)^2a_{k-2} \\
        \label{recurrence2} b_k &= (u+2t+3)b_{k-1}-(t+1)^2b_{k-2}.
    \end{align}
    with initial conditions $b_{-1}=0, a_0 = b_0 = 1, a_1 = u+3t+4$.

    We can fully describe the unnormalized steady state probabilities of the infinite family $\DASEP(n,2,2)$ as follows.
    When $n=2k+1$ is odd,
    $$
    \begin{array}{|c|c|c|}
    \hline
    \mu & \pi_{\DASEP}(\mu) \\ \hline
    S_n((1,1,0,\dots,0)) & a_k \\ \hline
    0\dots010^m 20\dots0 & ua_k + u(t-1)(t+1)^m a_{k-m-1}, (0\leq m<k) \\ \hline
    0\dots020^m 10\dots0 & ua_k - u(t-1)(t+1)^m a_{k-m-1}, (0\leq m<k) \\ \hline
    S_n((2,2,0,\dots,0)) & u^2 a_k \\ \hline
    \end{array}
    $$
    When $n=2k+2$ is even,
    $$\begin{array}{|c|c|c|}
    \hline
    \mu & \pi_{\DASEP}(\mu) \\ \hline
    S_n((1,1,0,\dots,0)) & b_k \\ \hline
    0\dots010^m 20\dots0 & ub_k + u(t-1)(t+1)^m b_{k-m-1}, (0\leq m \leq k) \\ \hline
    0\dots020^m 10\dots0 & ub_k - u(t-1)(t+1)^m b_{k-m-1}, (0\leq m \leq k) \\ \hline
    S_n((2,2,0,\dots,0)) & u^2 b_k \\ \hline
    \end{array}$$
\end{thm}

\begin{proof}
We will denote $\pi_{\DASEP}(\mu)$ by $p_\mu$ for simplicity in the proof.
We need to show that these formulae satisfy balance equation at each state. Let $p_w$ for any binary word $w$ be $x$, then $p_\mu$ for $\mu \in S_n((2,2,0,\dots,0))$ would be equal to $u^2x$ by \cref{thm:equal}. Let $p_{\dots 1 0^m 2\dots}$ be $q^+_m$ and $p_{\dots20^m1\dots}$ be $q^-_m$ for $m \in [0,k]$ or $[0,k)$. The following equations hold.
    \begin{align}
        \label{triv1} (2u+t+1)x &= (t+1)x + q^+_1 + q^-_1 \\
        \label{eq1} (u+2t+2)q^-_0 &= (t+1)q^-_1 + (u+u^2)x + q^+_0 \\
        \label{eq2} (u+2t+3)q^-_m &= (t+1) q^-_{m+1} + (u+u^2)x + (t+1) q^-_{m-1} \\
        \label{triv2} (t+3)u^2x &= (t+1)u^2x + u(q^+_0 + q^-_0)
    \end{align}
    When $n$ is odd, \cref{eq2} is true for $m \in [1,k-2]$, and we have an additional equation
    \begin{align}
        \label{oddinit} (u+2t+3)q^-_{k-1} &= (t+1) (q^-_{k-2}+q^+_{k-1}) + (u+u^2)x.
    \end{align}
    When $n$ is even, \cref{eq2} is true for $m \in [1,k-1]$, and we have an additional equation
    \begin{align}
        \label{eventop} (u+2t+3)q^-_k &= (t+1) (q^-_{k-1}+q^+_{k-1}) + (u+u^2)x.
    \end{align}
    
    By \cref{thm:main}, we have $q^+_i + q^-_i = 2ux$ for all $i$. Applying this to \cref{triv1}, \cref{triv2}, and \cref{eventop}, we see that all of them hold trivially.

    To show \cref{eq1}, we may assume $n$ is odd, since the even case would be the same. Deduce $u(u+t+3)x$ from both sides then divide both sides by $u(t-1)$, we are left with
    $$x - (u+2t+2)a_{k-1} = - (t+1)^2 a_{k-2} + a_{k-1}$$
    which is the same as 
    $$x = (u+2t+3) a_{k-1} - (t+1)^2 a_{k-2}. $$
    Since $x = a_k$, this is true by \cref{recurrence1}.

    To show \cref{eq2}, we rewrite the recurrence relations \cref{recurrence1} and \cref{recurrence2} into
    $$(t+1)(q^-_{m-1}-ux) = (u+2t+3)(q^-_m-ux) - (t+1) (q^-_{m+1}-ux),$$
    eliminating $x$ out of the equations, hence proving \cref{eq2}.

    To show \cref{oddinit}, we deduce $u(u+2t+3)x$ from both sides then divide both sides by $u(t-1)(t+1)^{m-1}$ so that we are left with
    $$-(u+2t+3) = -a_1 + t+1$$
    which is true since $a_1 = u+3t+4$.
\end{proof}

\begin{figure}
    \centering
    \includegraphics{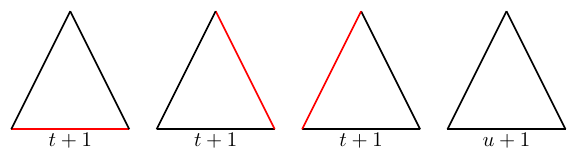}
    \caption{$a_1 = u+3t+4$}
    \label{C3}
\end{figure}

\begin{figure}
    \centering
    \includegraphics{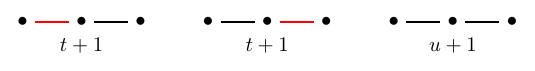}
    \caption{$b_1 = u+2t+3$}
    \label{L3}
\end{figure}

\begin{rema}
A matching in a graph is a set of pairwise non-adjacent edges.
Consider matchings $M$ in the cycle graph $C_{2k+1}$ or line graph $L_{2k+1}$ with $(2k+1)$ vertices. Assign each matching $M$ a weight of $(t+1)^{|M|}(u+1)^{k-|M|}$. 
Then $a_k$ is the sum of weights over all matchings of the cycle $C_{2k+1}$, i.e.,
\begin{equation}\label{eq:ak}
    a_k = \sum_{M: C_{2k+1}} (t+1)^{|M|}(u+1)^{k-|M|}
\end{equation}
and $b_k$ is that of the line $L_{2k+1}$, i.e.
\begin{equation}\label{eq:bk}
    b_k = \sum_{M: L_{2k+1}} (t+1)^{|M|}(u+1)^{k-|M|}.
\end{equation}
These can be seen via induction. The base cases can be seen in \cref{C3}, and \cref{L3} which equal to the first rows of \cref{tab:3&4}.
\end{rema}
\begin{proof}
We prove that the right hand side of ~\cref{eq:ak} satisfies the recurrence relation ~\cref{recurrence1} and the right hand side of ~\cref{eq:bk} satisfies ~\cref{recurrence2}.

    For the line graph $L_{2k+1}$, we label the edges by $[2k]:=\{1,2,\dots,2k\}$ from left to right. Take a matching $M$ in the first $[2k-2]$ edges (or, in the $L_{2k-1}$ subgraph), then $M, M\cup \{2k\}$ are both matchings in $L_{2k+1}$. However, $M \cup \{2k-1\}$ is only a matching if $2k-2 \notin M$, and if $2k-2 \in M$, then $M \setminus \{2k-2\}$ is a matching in $[2k-4]$ (or, on the $L_{2k-3}$ subgraph) because $2k-2 \in M$ implies that $2k-3 \notin M$. Therefore, $[(u+1)+2(t+1)] b_{k-1}$ counts $M, M\cup\{2k-1\},M\cup\{2k\}$ for all matchings $M$ in $L_{2k-1}$, and $(t+1)^2 b_{k-2}$ counts the those set of edges given by $M\cup\{2k-1\}$ that are not matchings.

    For the cycle graph $C_{2k+1}$, we label the edges by $[2k+1]$. The argument is very similar, but we have to subtract another copy of $(t+1)^2 a_{k-2}$ which comes from the possible non-matchings given by $M\cup \{2k+1\}$. However, if we take a matching $N$ on $\{1,\dots,2k-3\}$, then $N \cup \{2k-1, 2k+1\}$ is a matching in $C_{2k+1}$, which is counted by $(t+1)^2 a_{k-2}$ and added back.
\end{proof}

\section{Homomesy}
\cref{thm:main} can be viewed as a statement about taking average of some statistic over the orbit of a group acting on the particles, which is an instance of a phenomenon called \emph{homomesy} by Propp and Roby \cite{ProppRoby}. This phenomenon was first noticed by Panyushev \cite{dmitri} in 2007 in the context of the rowmotion action on the set of antichains of a root poset; Armstrong, Stump, and Thomas \cite{nestcross} proved Panyushev’s conjecture in 2011. 

\begin{defi}[\cite{ProppRoby}]
    Given a set $S$, an invertible map $\tau: S \to S$ such that each $\tau$-orbit is finite, and a function (or ``statistic") $f: S \to K$ taking values in some field $K$ of characteristic zero, we say the triple $(S, \tau, K)$ exhibits \emph{homomesy} if there exists a constant $c \in K$ such that for every $\tau$-orbit $O \subset S$
    $$\frac{1}{\# O} \sum_{x \in O} f(x) = c.$$
    In this situation we say that $f$ is \emph{homomesic} under the action of $\tau$ on $S$, or more specifically $c$-mesic.
\end{defi}

Although the original definition concerns the action of the cyclic group generated by a single map $\tau$, this can be generalized to the action of any finite group, as pointed out in \cite[Section 2.1]{roby2016}. We can also generalize the definition such that the statistic $f$ takes value in a ring of polynomials (or even Laurent polynomials) over some field of characteristic zero, and the average of this statistic is equal up to a monic monomial. We make the generalized definition more precise as follows.

\begin{defi}
    Given a set $S$, a finite group $G$ acting on $S$ with finite orbits, and a function $f: S \to R = K[x_1,\dots,x_n]$ for a field $K$ of characteristic zero, we say the triple $(S, G, R)$ exhibits \emph{homomesy} if there exists a polynomial $c \in R$ such that for every $G$-orbit $O \subset S$, there exist nonnegative integers $e^O_i \in \Z_{\ge 0}$ for $i=1,\dots,n$ such that
    $$\frac{1}{\# O} \sum_{x \in O} f(x) = \prod x_i^{e^O_i} c.$$
\end{defi}

\begin{coro}
    Let the symmetric group $S_q$ acts on the states of $\DASEP(n,p,q)$ by permuting the $q$ positive parts (the particles). Then the triple $(\Gamma^{p,q}_n, S_q, \Q[u,t])$ exhibits homomesy with statistic $\pi_{\DASEP}$ in the more general sense defined above.
\end{coro}

\begin{proof}
    The orbits of $S_q$ action on $\Gamma^{p,q}_n$ are $S_n^w(\lambda)$ for $(w,\lambda) \in \Omega^{p,q}_n$. The orbit size is 1 when $\lambda = \langle 1^q 0^{n-q} \rangle$. Let $c = \pi_{\DASEP}(w)$ for any binary word $w \in S_n(1^q 0^{n-q})$ and $c$ does not depend on the choice of binary word $w$ by (1) of \cref{thm:main}. 
    By (2) of \cref{thm:main}, we have
    $$\frac{1}{|S_n^w(\lambda)|} \sum_{\mu \in S_n^w(\lambda)} \pi_{\DASEP} = u^{|\lambda|-q} c.$$
\end{proof}

\begin{coro}
    Let the symmetric group $S_n$ acts on the states of $\DASEP(n,p,q)$ by permuting the $n$ sites. Then the triple $(\Gamma^{p,q}_n, S_n, \Q[u,t])$ exhibits homomesy with statistic $\pi_{\DASEP}$. 
\end{coro}

\begin{proof}
    The orbit of $S_n$ action are $S_n(\lambda)$ for $\lambda \in \chi^{p,q}$. Let $c = \pi_{\DASEP}(w)$ for any binary word $w$, and $c$ does not depend on the choice of $w$ by (1) of \cref{thm:main}. Then for $\lambda = \langle 1^q 0^{n-q} \rangle$, we have
    $$\frac{1}{|S_n(1^q 0^{n-q})|} \sum_{w} \pi_{\DASEP}(w) = c$$
    by definition of $c$. Taking $\mu = \langle 1^q 0^{n-q} \rangle$, by \cref{cor:restricted}, we have 
    $$\frac{1}{|S_n(\lambda)|} \sum_{\nu \in S_n(\lambda)} \pi_{\DASEP}(\nu) = u^{|\lambda|-q} c.$$
\end{proof}

\printbibliography

\end{document}